\documentclass{amsart}
\usepackage{tikz-cd}

\usepackage{amssymb}

\theoremstyle{plain}
\newtheorem{theorem}[equation]{Theorem}
\newtheorem{lemma}[equation]{Lemma}
\newtheorem{proposition}[equation]{Proposition}
\theoremstyle{definition}
\newtheorem{definition}[equation]{Definition}
\newtheorem{example}[equation]{Example}

\newtheorem{question}[equation]{Question}
\newtheorem{conjecture}[equation]{Conjecture}

\newtheorem*{theorem*}{Theorem}

\DeclareMathOperator{\Hom}{Hom}
\DeclareMathOperator{\End}{End}
\DeclareMathOperator{\Ext}{Ext}
\DeclareMathOperator{\rad}{rad}
\DeclareMathOperator{\op}{op}
\DeclareMathOperator{\id}{id}
\DeclareMathOperator{\image}{Im}
\DeclareMathOperator{\modu}{mod}
\DeclareMathOperator{\height}{ht}
\DeclareMathOperator{\add}{add}
\DeclareMathOperator{\twmod}{twmod}
\DeclareMathOperator{\ess}{ess}

\begin{document}

\title{Towards bound quivers for exact categories}

\author{Julian K\"ulshammer}


\address{Julian K\"ulshammer\\
Department of Mathematics
University of Uppsala \\ Box 480 \\ Uppsala,
Sweden} \email{julian.kuelshammer@math.uu.se}

\thanks{The author would like to thank Ren\'e Marczinzik for asking a question that led to the inclusion of Lemma \ref{quadratic-a-infinity}. He would also like to thank Volodymyr Mazorchuk for the permission of including Lemma \ref{walter's-lemma} and for answering his questions about Example \ref{category-o-example}. Furthermore he would like to thank the anonymous referee for comments regarding the literature.}

\keywords{Exact category, resolving subcategory, quasi-hereditary algebra, A-infinity algebra, Gabriel quiver, BGG category}


\begin{abstract}
In this survey article we propose the notion of a bound quiver for an exact category generalising the classical concept of the Gabriel quiver and its relation for a module category as certain ring extension. The notion is motivated by joint work of the author with Vanessa Miemietz, Steffen Koenig and Sergiy Ovsienko in the case where the exact category is the category of modules with a standard filtration for a quasi-hereditary algebra. We conclude with a variety of examples in this case. 
\end{abstract}

\maketitle


\section{Introduction}

In the representation theory of finite-dimensional algebras, the notion of a bound quiver, consisting of the Gabriel quiver of the algebra and a set of relations generating an admissible ideal have been an essential tool, not least because of the easy access to constructing families of combinatorial examples with nice homological properties as well as counterexamples with particularly complicated homological behaviour. 

For every finite-dimensional algebra, there is a Morita equivalent algebra with each indecomposable projective (up to isomorphism) appearing exactly once in the regular module, its basic representative. The basic representative is unique in the sense that two finite-dimensional algebras are  Morita equivalent if and only if their basic representatives are isomorphic. Over an algebraically closed field $\Bbbk$, Gabriel's structure theorem asserts that its basic representative is isomorphic to the quotient of a path algebra of a finite quiver $Q$ by an admissible ideal $I$. Stasheff's notion of an $A_\infty$-algebra provides a homotopical perspective on this theorem. Roughly speaking, an $A_\infty$-algebra is a differential graded algebra which is not associative on the nose but only up to a system of higher homotopies. A particular example is the $\Ext$-algebra of any module $M$. If $M=L$ is the direct sum of the simple $A$-modules, then it is well-known that the space $\Ext^1_A(L,L)$ is dual to the space spanned by the arrows of the Gabriel quiver of $A$ while the space $\Ext^2_A(L,L)$ is dual to a space spanned by minimal generators of the admissible ideal of relations. In this case, the $A_\infty$-structure restricts to maps 
\[m_n\colon (\Ext^1_A(L,L))^{\otimes n}\to \Ext^2_A(L,L),\] 
which are dual to a section of the projection $\Bbbk Q\to \Bbbk Q/I$, an instance of what is nowadays often called $A_\infty$-Koszul duality. A complicated way to deduce uniqueness of the basic representative of a finite-dimensional algebra would be to deduce it from Kadeishvili's theorem on the uniqueness of the $A_\infty$-structure on $\Ext^*_A(L,L)$ up to $A_\infty$-isomorphism. In addition to Morita theory and Gabriel's structure theorem, another foundational theorem in representation theory of finite-dimensional algebra is the Wedderburn--Malcev theorem stating that over a perfect field, the algebra $A/\rad(A)$ embedds as a maximal semisimple subalgebra into $A$. Moreover, this embedding is unique up to conjugation with an invertible element of $A$. 

Extension-closed subcategories of module categories of finite-dimensional algebras have been studied for a long time. An example of such an extension-closed subcategory is given by the category $\mathcal{F}(\Delta)$ of modules admitting a standard filtration for the class of quasi-hereditary algebras. These algebras and the related highest weight categories were introduced by Cline, Parshall, and Scott to abstract certain properties arising in the representation theory of Lie algebras. Important examples are blocks of BGG category $\mathcal{O}$, Schur algebras, but also hereditary algebras and Auslander algebras of representation-finite algebras. Motivated in particular by the example of BGG category $\mathcal{O}$ in \cite{Koe95}, Koenig introduced the notion of an exact Borel subalgebra of a quasi-hereditary algebra. In recent years, it was realised that a variant of his notion actually makes sense more generally for exact categories of the form $\mathcal{F}(M)$ of modules admitting a filtration by a particular class of modules. For such an exact category, we suggest to call its bound quiver an algebra extension $A\subseteq R$ such that $A$ is basic, $R\otimes_A A/\rad(A)$ is isomorphic to the direct sum of modules under consideration, and the algebra extension is normal (a technical condition that will be explained in the main body of the text), and regular exact in the sense that $R$ is projective as a right $A$-module and the induced maps
\[\Ext^m_A(A/\rad(A),A/\rad(A))\to \Ext^m_R(R\otimes_A A/\rad(A), R\otimes_A A/\rad(A))\]
are isomorphisms for all $m\geq 1$. Before describing our results in the quasi-hereditary case, let's discuss two extremal cases. If $M$ is the collection of simple modules, then $\mathcal{F}(M)$ is the whole module category, and the algebra extension $A\subseteq A$ is a bound quiver for any basic algebra $A$. If, on the other hand, $M$ is the collection of indecomposable projective modules, then $\mathcal{F}(M)$ is the category of finite-dimensional projective modules and if the ground field is perfect then, any choice of maximal semisimple subalgebra $L$ of $A$ gives a bound quiver $L\subseteq A$. 

The main theorems of our joint works with Vanessa Miemietz, Steffen Koenig and Sergiy Ovsienko so far provide an analogue of the Wedderburn--Malcev theorem for standard modules instead of projective modules:

\begin{theorem*}
Let $\Lambda$ be a quasi-hereditary algebra over an algebraically closed field $\Bbbk$. Then there exists a bound quiver $A\subseteq R$ for the collection of standard modules such that $R$ is Morita equivalent to $\Lambda$ as a quasi-hereditary algebra. Moreover, this bound quiver is unique in the sense that if $A\subseteq R$ and $B\subseteq S$ are two bound quivers for $\Lambda$, then there exists an algebra isomorphism from $R$ to $S$, which sends $A$ to $B$.  
\end{theorem*}

It is currently open whether if $R=S$, the algebra automorphism can be given by conjugation by an invertible element. The existence part of this theorem has been generalised to the case of (pre-)standardly stratified algebras and the collection of (proper) standard modules by Bautista, P\'erez, and Salmer\'on \cite{BPS23} as well as Goto \cite{Got22}. Note that in the latter case, in contrast to the other cases the subalgebra $A$ is  no longer quasi-hereditary. We expect a corresponding uniqueness result in this case and more generally, existence and uniqueness to be true in much larger generality.

In the known cases, the theorem is obtained by studying the $A_\infty$-structure on the $\Ext$-algebra of the standard modules. Applying a truncated form of $A_\infty$-Koszul duality yields a differential graded structure on a tensor algebra $T_A(\overline{V})$. By a theorem of Roiter, which can be seen as a special case of a relative Koszul duality, such a datum corresponds to an $A$-coring structure on $A\oplus \overline{V}$ whose dual is the algebra extension $A\subseteq R$. 

In the cases of projective modules and simple modules above it is true that for a bound quiver $A\subseteq R$, both algebras $A$ and $R$ are in fact basic. This is usually not true in the case of standard modules for a quasi-hereditary algebra. In fact, Conde proved that $R$ being basic is equivalent to the unusual condition that the Jacobson radical of all standard modules is filtered by the costandard modules. She also provided a recursive formula to compute the decomposition multiplicities of $R$. Several examples of bound quivers have been computed recently. In particular, we mention results by Thuresson for dual extension algebras \cite{Thu22} and path algebras of Dynkin type $\mathbb{A}$, see \cite{Thu24}, and results by Rodriguez Rasmussen on skew group algebras of quasi-hereditary algebras, see \cite{Rod23}. In this survey, we add to this list by computing more examples of hereditary algebras as well as principal blocks of BGG-category $\mathcal{O}$ for dihedral groups. 

The article is organised as follows. In Section \ref{bound-quiver-abelian} we recall the theory of bound quivers for abelian categories with an emphasis of those aspects that generalise to the proposed setting of exact categories. In particular, we describe the perspective of $A_\infty$-structures on $\Ext^*_A(A/\rad(A),A/\rad(A))$ as a means to understand the presentation of an algebra by quiver and relations. We resume in Section \ref{exact-cats} with recalling the notion of an exact category and surveying on some classical and recent results regarding categories of modules with a filtration and simple objects in exact categories. In Section \ref{regular-exact-subalgebras} we propose our notion of a regular exact subalgebra and how it is motivated by the theory of theory of regular bocses or corings. In Section \ref{qh-case} we study in detail the case of quasi-hereditary algebras, where the proposed notion has been worked out to a large extent. In particular we state the main theorems of our collaborations with Steffen Koenig and Sergiy Ovsienko as well as with Vanessa Miemietz. We also highlight recent results of Teresa Conde, which is helpful to obtain partial information without using the $A_\infty$-machinery at all. The article concludes with Section \ref{examples} surveying recent results on particular classes of quasi-hereditary algebras as well as giving new examples of quasi-hereditary algebras with exact Borel subalgebras.  

Throughout, let $\Bbbk$ be a field. All algebras, modules, categories, etc. are assumed to be $\Bbbk$-linear, and usually finite-dimensional unless noted otherwise. By convention, if the side of the action is not specified, module means left module. 

\section{Bound quivers for abelian categories}\label{bound-quiver-abelian}

Abelian categories were introduced by Buchsbaum \cite{Buc55} and Grothendieck \cite{Gro57} for the purpose of unifying different cohomology theories. By their definition of being additive categories admitting all kernels and cokernels such that Noether's first isomorphism theorem holds, they are the first abstract notion to study extensions of objects given by short exact sequences. Important instances of abelian categories are provided by module categories of rings as well as categories of sheaves of abelian groups. The former is often taken as the prototypical example with proofs only given in the case of modules and then invoking Freyd--Mitchell's embedding theorem \cite[Theorem 4.4]{Mit64} that every (small) abelian category embedds as a full extension-closed subcategory of the module category over some ring. While this makes general abelian categories more accessible, module categories over general rings can still be quite complicated to study from a representation-theoretic perspective, for example in cases where the Jordan--H\"older theorem and the Krull--Remak--Schmidt theorem don't hold. This is why in the context of the `International Conference on Representations of Algebras' one often restricts to categories of finite-dimensional modules over finite-dimensional algebras or related categories.

\begin{definition}
An abelian category is called a \textbf{length category} if it is essentially small and every object has a finite composition series. 
\end{definition}

The following theorem is well-known, see e.g. \cite[(8.2)]{Gab73}, \cite[Corollary 2.17]{Del90}, and \cite[Theorem 2.11]{Paq18}.

\begin{theorem}
Let $\mathcal{A}$ be a $\Bbbk$-linear $\Hom$-finite abelian length category. Then the following are equivalent:
\begin{enumerate}
\item $\mathcal{A}$ admits a generator, i.e. an object $X$ such that every object of $\mathcal{A}$ is a quotient of a direct sum of copies of $\mathcal{A}$. 
\item $\mathcal{A}$ has a projective generator.
\item $\mathcal{A}$ has finitely many simple objects and enough projectives.
\item $\mathcal{A}$ has finitely many simple objects, is $\Ext^1$-finite, and there is a common bound for the Loewy length of all objects.
\item $\mathcal{A}$ is equivalent to the category of finite-dimensional modules over a finite-dimensional algebra. 
\end{enumerate}
\end{theorem}

Of course, the algebra $A$ is not unique, but only unique up to Morita equivalence. In order to obtain a uniqueness up to isomorphism, one has to restrict to the class of basic algebras, c.f. \cite[(2.1)]{NS43}, as noted by Morita \cite[Theorem 7.5]{Mor58}.

\begin{theorem}
For every finite-dimensional algebra $A$, there exists a basic algebra $B$, which is Morita equivalent to $A$. Moreover, this basic representative is unique up to isomorphism. 
\end{theorem} 

Of course there are many drawbacks when passing from an arbitrary finite-di\-men\-sio\-nal algebra to the corresponding basic algebra, depending on the context. For example, one cannot make use of the tensor structure in the case of modules over a Hopf algebra, one loses numerical information about the dimension of simple modules, and also algebraic information like the existence of certain subalgebras. However, there are also a lot of positive aspects to it, not least that one has access to the `combinatorial' object of a bound quiver for describing this basic representative. Such a description is provided by Gabriel's structure theorem, see \cite[(8.4)]{Gab73}. 

\begin{theorem}
Let $\Bbbk$ be algebraically closed. Then for every basic finite-dimensional algebra $A$ there is a finite quiver $Q$ and an admissible ideal $I$ such that $A\cong \Bbbk Q/I$.  
\end{theorem}

The classical construction of $Q$ to prove this theorem goes as follows: The set of vertices $Q_0$ of the quiver $Q$ is an indexing set for the isomorphism classes of indecomposable projective modules. For two vertices $\mathtt{i}, \mathtt{j}\in Q_0$, the arrows $\mathtt{i}\to \mathtt{j}$ index a basis of the space of irreducible maps $\rad (P_\mathtt{i}, P_\mathtt{j})/\rad^2(P_\mathtt{i}, P_\mathtt{j})$. Sending each arrow to a corresponding basis element of $A$ then defines a surjective algebra homomorphism $\Bbbk Q\twoheadrightarrow A$ whose kernel is an admissible ideal. The assumption on the ground field can be replaced by the assumption on the algebra that $A/\rad A$ is isomorphic to a product of $\Bbbk$. The theorem generalises to perfect fields (or more generally to algebras $A$ such that $A/\rad A$ is separable). We refer the reader to \cite[(4.1)]{B91} and \cite{Ber11} for this generalisation. One may wonder whether it is alternatively possible to use information about the simple objects as they seem to be at least equally important in the definition of a finite abelian category (even more so for a length category). A version of this has been known for a long time and goes back to Govorov \cite{Gov73}, see also \cite{Bon83}:

\begin{lemma}\label{ext-quiver}
Let $A=\Bbbk Q/I$ for a finite quiver $Q$ and an admissible ideal $I$. Then 
\begin{itemize}
\item $\mathbb{D}\Ext^1(L_\mathtt{i},L_\mathtt{j})\cong e_\mathtt{j} Q_+/Q_+^2e_\mathtt{i}$,
\item $\mathbb{D}\Ext^2(L_\mathtt{i},L_\mathtt{j})\cong e_\mathtt{j} I/(IQ_++Q_+I) e_\mathtt{i}$.
\end{itemize}
\end{lemma}

In other words, the lemma says that the dimension of the first extension group between two simples is given by the number of arrows between the corresponding vertices while the dimension of the second extension group between two simples is given by the number of generators from $\mathtt{i}$ to $\mathtt{j}$ in a minimal generating set of the admissible ideal $I$. This numerical coincidence suggests that there is a way to construct the basic algebra from the $\Ext$-spaces between the simple modules. Unfortunately, life is not simple, and even the (Yoneda) algebra structure on the $\Ext$-groups does not suffice as the following examples shows:

\begin{example}
Let $A=\Bbbk[t]/(t^\ell)$ with $\ell\geq 3$. Then $\Ext^*_A(\Bbbk, \Bbbk)\cong \Bbbk[x,y]/(x^2)$ as graded algebras (with $x$ of degree one and $y$ of degree $2$), independently of the parameter $\ell$, see \cite[Example B.2.2]{Mad02}. 
Analogously, if $A=\Bbbk\mathbb{A}_n/\rad^\ell$, the path algebra of the linearly oriented type $A$ quiver modded out by all paths of length $\ell$, then 
\[\Ext^*_A(A/\rad A,A/\rad A)\cong \begin{tikzcd}\mathtt{1}\arrow{r}{x_\mathtt{1}}\arrow[bend left]{rr}{y_\mathtt{1}} &\mathtt{2}\arrow{r}{x_\mathtt{2}}\arrow[bend left]{rr}{y_\mathtt{2}}&\mathtt{3}\arrow{r}{x_\mathtt{3}}&\dots\arrow{r}{x_{\mathtt{n-1}}}&\mathtt{n}\end{tikzcd}\] 
with relations $x_{\mathtt{i+1}}x_{\mathtt{i}}$, where all $x_\mathtt{i}$ are of degree $1$ and all $y_\mathtt{i}$ are of degree $2$, see \cite[Section 7]{Thu22}. 
\end{example}

More generally, in \cite{GZ94}, Green and Zacharia computed the algebra structure on $\Ext^*_A(A/\rad A, A/\rad A)$ for every monomial algebra $A$ and in this way many more such examples can be constructed. There is however a way to reconstruct $A$ from $\Ext^*_A(A/\rad A, A/\rad A)$, but it needs higher homotopical information, more precisely its $A_\infty$-algebra structure as defined by Stasheff in \cite{Sta63}.

\begin{definition} 
An \textbf{$A_\infty$-algebra} $\mathcal{E}$ is a graded vector space $\mathcal{E}$ together with linear maps $m_n\colon \mathcal{E}^{\otimes n}\to \mathcal{E}$, homogeneous of degree $2-n$, for $n\geq 1$, such that for all $n\geq 1$ the following equation holds:
\[\sum_{\substack{n=r+s+t\\r,t\geq 0, s\geq 1}}(-1)^{r+st} m_{r+1+t} (\id^{\otimes r}\otimes m_s\otimes \id^{\otimes t})=0,\]
where these formulas and subsequent ones are to be understood with the Koszul sign rule implicit: $(f\otimes g)(x\otimes y)=(-1)^{|g|\cdot |x|} f(x)\otimes g(y)$ for homogenous maps $f$ and $g$ and homogenous elements $x$ and $y$. 
\end{definition}

To understand the formulas, it is enlightening to consider low values of $n$. For $n=1$, we obtain a degree $1$ map $m_1$ such that $m_1m_1=0$, i.e. $\mathcal{E}$ can be considered as a cochain complex. The case $n=2$ yields the degree $0$ map $m_2$ satisfying $m_2(\id\otimes m_1+m_1\otimes \id)=m_1m_2$, which is precisely the fact that the multiplication $m_2$ satisfies the graded Leibniz rule with respect to the differential $m_1$. Now for $n=3$, we see that 
\[m_2(\id\otimes m_2-m_2\otimes \id)=m_1m_3+m_3\underbrace{(\id^{\otimes 2}\otimes m_1+\id\otimes m_1\otimes \id+m_1\otimes \id)}_{m_1^{\otimes}}.\] Oberving that $m_1^{\otimes}$ is precisely the differential on the third tensor power of $\mathcal{E}$, we see that $m_2$ is in general not associative, but only associative up to the homotopy $m_3$. The higher $m_n$ provide a coherent system of higher homotopies. While it is convenient to have the flexibility to work with general $A_\infty$-algebras, we are mostly interested in algebras where $m_2$ is in fact associative. The first situation where this happens is when $m_n=0$ for all $n\geq 3$. This gives the (non-full) subcategory of differential graded (dg) algebras. The second situation occurs when $m_1=0$. In this case, the $A_\infty$-algebras are called \textbf{minimal}. In fact, either of these cases suffices to desribe all $A_\infty$-algebras up to $A_\infty$-quasi-isomorphism in the following sense:

\begin{definition}
Let $\mathcal{A}$ and $\mathcal{B}$ be $A_\infty$-algebras. An \textbf{$A_\infty$-morphism} $f\colon \mathcal{A}\to \mathcal{B}$ consists of maps $f_n\colon \mathcal{A}^{\otimes n}\to \mathcal{B}$, homogeneous of degree $1-n$ such that for all $n\geq 1$ the following equations hold:
\begin{equation*}
\begin{split}\sum_{n=r+s+t} (-1)^{r+st} f_{r+1+t}(\id^{\otimes r}\otimes &f_s\otimes \id^{\otimes t})\\
&=\sum_{k, j_1,\dots,j_k} (-1)^{\sum_{u=1}^{k}(k-u)(j_u-1)} m_k(f_{j_1}\otimes \dots\otimes f_{j_k}).
\end{split}
\end{equation*}
An $A_\infty$-morphism is called an \textbf{$A_\infty$-quasi-isomorphism} if $f_1$ is a quasi-isomorphism of complexes. 
\end{definition} 

Again, the reader is encouraged to consider the lower degree cases of $n$ to see that $f_1$ is a homomorphism of complexes and compatible with the multiplications $m_2$ up to the homotopy $f_2$. A note of warning is that in this note we have decided not to go into the details of how to define unitality and augmentation for $A_\infty$-algebras. However, for some of the theorems below these properties are essential. The following theorem provides the existence of a dg model for each $A_\infty$-algebra, which can be convenient to work with since dg algebras only have two non-zero maps attached to them, namely $m_1$ and $m_2$. 

\begin{theorem}[{cf. \cite[p. 9]{K01}}]
For every $A_\infty$-algebra $\mathcal{A}$, there is a differential graded algebra $\mathcal{D}$ together with an $A_\infty$-quasi-isomorphism $g\colon \mathcal{A}\to \mathcal{D}$. 
\end{theorem}

However, even if $\mathcal{A}$ is finite-dimensional, the general existence result only provides an infinite-dimensional $\mathcal{D}$ (roughly speaking the tensor algebra over the tensor coalgebra of $\mathcal{A}$) that in practice will be hard to work with. Therefore, minimal $A_\infty$-algebras seem to be much closer in spirit to representation theory of quivers with relations as they provide the representative of an $A_\infty$-quasi-isomorphism class of lowest dimension. The existence of such minimal models was first proven by Kadeishvili in \cite{Kad82}.

\begin{theorem}
For every $A_\infty$-algebra $\mathcal{A}$, there is a minimal $A_\infty$-algebra $\mathcal{E}$ together with an $A_\infty$-quasi-isomorphism $f\colon \mathcal{E}\to \mathcal{A}$. In fact, $\mathcal{E}$ is even unique up to $A_\infty$-isomorphism.
\end{theorem}

The uniqueness up to $A_\infty$-isomorphism is in fact not that surprising as every $A_\infty$-quasi-isomorphism between minimal $A_\infty$-algebras is in fact an $A_\infty$-isomorphism (which is equivalent to $f_1$ being an isomorphism). We are now in a position to equip $\Ext^*_A(M,M)$ with the structure of an $A_\infty$-algebra for every $A$-module $M$. Let $P^*$ be a projective resolution of $M$. Then $\mathcal{D}=\End_A(P^*)$, the dg endomorphism algebra of $P^*$ is a dg algebra with differential 
\[\delta(f)=d_P\circ f-(-1)^{|f|}f\circ d_P\]
on homogeneous elements $f$. Its minimal model is isomorphic to $\Ext^*_A(M,M)$ as a graded algebra, but in general it has a lot of higher multiplications $m_n$. So far, we have given the minimal model as a pure existence result, but in fact for a dg algebra $\mathcal{D}$ there is Merkulov's construction to produce it. For this proceed with the following steps. Let $\mathcal{D}$ be a dg algebra. Choose a graded complement $\mathcal{L}$ of the cycles $\mathcal{Z}$ of $\mathcal{D}$. Then choose a graded complement $\mathcal{H}$ of the boundaries $\mathcal{B}$ inside the cycles $\mathcal{Z}$. The differential $d$ of $\mathcal{D}$ can thus be viewed as an isomorphism $d\colon \mathcal{L}\to \mathcal{B}$. Define $h\colon \mathcal{D}\to \mathcal{D}$ via $h=0$ on $\mathcal{H}\oplus \mathcal{L}$ and $h=d^{-1}$ on $\mathcal{B}$. Inductively defining 
\[\lambda_n=\sum_{\substack{r+s=n\\r,s\geq 1}}(-1)^{r-1}\lambda_2(h\lambda_r\otimes h\lambda_s)\]
where $\lambda_2$ is the multiplication on $\mathcal{D}$ and $h\lambda_1=\id$ by convention yields maps $\lambda_n\colon \mathcal{D}^{\otimes n}\to \mathcal{D}$, sometimes called the $\lambda$-kernels. Setting $\mathcal{E}=H^*(\mathcal{D})$ and choosing an isomorphism $\mathcal{E}\to \mathcal{H}$ yields an inclusion $i\colon \mathcal{E}\to \mathcal{D}$ and a projection $p\colon \mathcal{D}\to \mathcal{E}$. Finally, the maps $m_n\colon p\lambda_n i^{\otimes n}$ equip $\mathcal{E}$ with the structure of an $A_\infty$-algebra. 

We illustrate the construction with the first example of a biserial algebra, which is not special biserial, see \cite[p. 175]{SW83}.

\begin{example}
Let $A$ be the quotient of the path algebra of 
\[
\begin{tikzcd}
\mathtt{1}\arrow{r}{\alpha}&\mathtt{2}\arrow{rr}{\beta}\arrow{rd}{\gamma}&&\mathtt{4}\arrow{r}{\varepsilon}&\mathtt{5}\\
&&\mathtt{3}\arrow{ru}{\delta}
\end{tikzcd}
\]
with relations $(\beta-\delta\gamma)\alpha$ and $\varepsilon\beta$. This algebra is of global dimension $2$, thus the algebra $\Ext^*_A(L,L)$ is (apart from the identity morphisms) concentrated in degrees $1$ and $2$. Each arrow corresponds to an element of $\Ext^1$ while the two relations correspond to elements of $\Ext^2$. One can check that in the $\Ext$-algebra the elements corresponding to $\alpha$ and $\beta$ (resp. $\beta$ and $\varepsilon$) compose to these elements of $\Ext^2$. However, the result of the $A_\infty$-structure depends on the choice of representative for the elements of $\Ext^1$ in the dg algebra $\mathcal{D}=\End_A(P^*)$. In particular, the extension corresponding to the arrow $\gamma\colon \mathtt{2}\to \mathtt{3}$ can be lifted in two entirely different ways:
\[
\begin{tikzcd}
0\arrow{r}&P(\mathtt{5})\arrow{d}{g}\arrow{r}&P(\mathtt{3})\oplus P(\mathtt{4})\arrow{d}{(\id, f)}\arrow{r}&P(2)\to 0\\
0\arrow{r}&P(\mathtt{4})\arrow{r}&P(\mathtt{3})\arrow{r}&0
\end{tikzcd}
\] 
We can choose to represent it with $(g,f)=(0,0)$ or $(g,f)=(\rho_\delta,\rho_\varepsilon)$, where $\rho$ denotes right multiplication by the element in the index. With the first choice, one has that $i(\gamma)i(\alpha)\neq 0$, but $i(\delta)i(\gamma)=0$, yielding $m_3(\delta,\gamma,\alpha)\neq 0$. With the second choice, one has $i(\gamma)i(\delta)=0$, but $i(\delta)i(\gamma)\neq 0$, yielding $m_3(\varepsilon,\delta,\gamma)\neq 0$. This doesn't contradict uniqueness of the $A_\infty$-structure up to isomorphism because the existence of the sequence of maps $f_n$ makes this notion very flexible. We will soon see that this behaviour is to be expected since the algebra $A$ is isomorphic to the algebra $B$ with the same quiver but relations $\beta\alpha$ and $\varepsilon(\beta-\delta\gamma)$.
\end{example}

For the remainder of the section let $A$ be a (basic) algebra such that $L:=A/\rad A\cong \prod_{\mathtt{i}=\mathtt{1}}^\mathtt{n}\Bbbk$, so that $A$ is isomorphic to the quotient of a path algebra of a finite quiver by an admissible ideal. We now want to give an alternative way to construct this quiver and relations, in this case starting with the simples instead of the projectives, but making use of $A_\infty$-algebras. We start with the following lemma

\begin{lemma}[{\cite[Proposition 1 b)]{K02}}]\label{ext-1-generation}
As an $A_\infty$-algebra, $\Ext^*_A(L,L)$ is generated in degree $1$ (and strictly speaking also in degree zero, consisting only of linear combinations of the identities, by Schur's lemma). 
\end{lemma}

In particular, the lemma says that the following maps are essential to understand the $A_\infty$-structure on the $\Ext$-algebra:
\[m_n\colon (\Ext^1_A(L,L))^{\otimes n}\to \Ext^2_A(L,L)\]
Remembering that the arrows of the quiver, respectively the relations, of the algebra correspond to basis vectors of the dual space of $\Ext^1$, respectively $\Ext^2$, we dualise these maps and obtain maps
\[m_n^\#\colon \mathbb{D}\Ext^2_A(L,L)\to (\mathbb{D}(\Ext^1(L,L)))^{\otimes n},\]
using that standard duality commutes with taking tensor products as the spaces involved are finite-dimensional. Thus, these maps give us a way,  starting from an element of the dual of $\Ext^2$ to produce a linear combination of paths in the quiver. This is the promised construction of the quiver and relations of an algebra from the $\Ext$-algebra of the simples. 

\begin{theorem}[{\cite[Proposition 2]{K02}, cf. \cite[Claim 1.1]{Seg08}}]
The basic algebra $A$ is isomorphic to the quotient of the tensor algebra of $\mathbb{D}\Ext^1_A(L,L)$ over the semisimple algebra $L$ by the ideal generated by $\image\sum m_n^\#$. 
\end{theorem}

This theorem sometimes goes under the name of \textbf{$A_\infty$-Koszul duality}. There is in fact an inverse to the above construction that we will sketch. Note that there are isomorphisms 
\[\bigoplus_{j=1}^\infty \Ext^1_A(L,L)^{\otimes j}\xrightarrow{\sim} \mathbb{D}Q_+\]
and 
\[\Ext^2_A(L,L)\cong \mathbb{D}(I/Q_+I+IQ_+).\]
Thus, the choice of a section, $I/(Q_+I+IQ_+)\to I\to Q_+$, in other words choosing representatives for the relations, yields a dual map 
\[\bigoplus_{j=1}^\infty \Ext^1_A(L,L)\to \Ext^2_A(L,L),\]
which is a candidate for the restriction of the $A_\infty$-structure to this particular degree. And this candidate is in fact part of an $A_\infty$-structure:

\begin{theorem}[{\cite[Proposition 2]{K02}}]
Each such map is the sum of $m_n$ for some choice of minimal model $A_\infty$-structure on $\Ext^*_A(L,L)$.
\end{theorem}

Unfortunately, the given reference doesn't contain a proof. We refer the reader to \cite[Theorem A]{LPWZ09} or \cite[Theorem B]{KM22} for proofs in analogous situations. We have seen that in general, the $A_\infty$-structure is necessary to reconstruct the algebra $A$ from the $\Ext$-algebra of the simples. There are however situations in which the algebra structure on $\Ext^*_A(L,L)$ suffices. Recall that an algebra $A=\Bbbk Q/I$ is called \textbf{quadratic} if the ideal $I$ is generated by linear combinations of paths of length two.  In particular, $I$ is homogeneous and thus $A$ is graded by path lengths. As for any graded algebra, the simples and the projective modules are gradable, see e.g. \cite{GG82}. In particular, the $\Ext$-algebra of the simples affords another grading (in addition to the homological grading), sometimes called the internal grading or the Adams grading. All the previously mentioned constructions have graded analogues such that all the maps, like the multiplications $m_n$, the $A_\infty$-morphisms $f_n$ and the $\lambda$-kernels $\lambda_n$, have internal degree $0$. In the case that $A$ is quadratic, the spaces $\Ext^1_A(L,L)$, resp. $\Ext^2_A(L,L)$ are concentrated in internal degree $1$ and $2$, respectively, see e.g. \cite[Corollary 5.3]{PP05}. The following is an immediate corollary:

\begin{lemma}\label{quadratic-a-infinity}
If $A$ is quadratic and the maps 
\[m_n\colon \Ext^1_A(L,L)^{\otimes n}\to \Ext^2_A(L,L)\]
are chosen so that they have internal degree $0$, then these maps vanish except possibly for $n=2$. 
\end{lemma}

In particular, instead of taking the quotient by $\image\sum m_n^\#$ in $A_\infty$-Koszul duality, one can just take $\image m_2^\#$ to reconstruct the algebra from the $\Ext$-algebra of the simples (without needing its $A_\infty$-structure). A particularly nice case of this is when $A$ is a Koszul algebra, i.e. $\Ext^n_A(L,L)$ is concentrated in internal degree $n$ for all $n$.  

\begin{lemma}[{\cite[Proposition 1 a)]{K02}, \cite[Corollary V.0.6]{Con11}}]
The graded algebra $A$ is Koszul if and only if any $A_\infty$-structure on the graded algebra $\Ext^*_A(L,L)$ satisfies $m_n=0$ except possibly for $n=2$.
\end{lemma}
 
Note that the difference between quadratic and Koszul in terms of the $A_\infty$-structure is that for Koszul algebras, $m_n=0$ for $n\neq 2$ for all inputs, while for quadratic algebras, $m_n=0$ for $n\neq 2$ when restricted to $\Ext^1_A(L,L)^{\otimes n}$. 

Although the choice of presentation determines the restriction of an $A_\infty$-structure to $\Ext^1_A(L,L)$, and thus are quite well understood, in general there are very few algebras for which the $A_\infty$-structure on the full $\Ext^*_A(L,L)$ is known. A larger class for which the computation has been done are the aforementioned monomial algebras, see \cite[Theorem 4.7]{Tam21}. As an application of this computation, Dotsenko, Gelinas, and Tamaroff were able to prove that a monomial algebra satisfies the (Fg)-condition if and only if it is Gorenstein, see \cite{DGT23}. Thus, a challenge for the community would be as follows.

\begin{question}
For well-understood classes of finite-dimensional algebras such as (special) biserial algebras, give a combinatorial description of the $A_\infty$-structure on the $\Ext$-algebra of the simple modules.  
\end{question}

The reader who wants to learn more about $A_\infty$-algebras and their use in representation theory is referred to Keller's survey article for the ICRA 2000 in \cite{K02}. Similar surveys, covering slightly different aspects are \cite{K01}, \cite{K06}, and \cite{LPWZ04}. For a more extensive introduction to the theory of $A_\infty$-algebras, we recommend \cite{Boc21} and \cite{Sei08}. The view-point of non-commutative deformation theory on this matter, which is not covered in this survey can be found in \cite{Seg08} and \cite{Boo22}. As usual, master and PhD theses can provide a great source of exposition containing more details than other sources, we refer the reader for example to \cite{L-H03}, \cite{Con11}, and \cite{Wal21}.

\section{Exact categories}\label{exact-cats}

Sometimes in studying abelian categories one is interested in particular extension-closed subcategories -- either because they contain all the information one is interested in or because the entire abelian category is too complicated to fully understand. Examples include categories of vector bundles, categories of projective or Gorenstein projective modules, monomorphism categories, modules of constant Jordan type, or categories of filtered modules for quasi-hereditary algebras. In the context of algebraic $K$-theory, Quillen introduced in \cite{Qui73} the notion of an exact category to abstract properties of such categories. For an introduction to exact categories, we refer the reader to \cite{Bue10} and \cite{FS10}.

\begin{definition}
An \textbf{exact structure} on an additive category $\mathcal{C}$ consists of a choice of kernel-cokernel pairs, called \textbf{conflations}, where the kernels are called \textbf{inflations} and the cokernels are called \textbf{deflations}, such that the following axioms hold:
\begin{enumerate}
\item[(Ex0)] The class of conflations is closed under isomorphisms and contains the split exact sequences.
\item[(Ex1)] Compositions of deflations are deflations.
\item[(Ex$1^{\op}$)] Compositions of inflations are inflations.
\item[(Ex2)] Pullbacks of deflations along arbitrary morphisms exist and are again deflations.
\item[(Ex$2^{\op}$)] Pushouts of inflations along arbitrary morphisms exist and are again inflations. 
\end{enumerate}
The additive category together with its exact structure is called an \textbf{exact category}.
\end{definition}

The Gabriel--Quillen embedding theorem shows that every (small) exact category can be embedded as a full extension-closed subcategory in an abelian category. In practice, many of the exact categories appearing in representation theory satisfy additional properties. For example, they are closed under direct summands. Often they also have enough projectives in the sense that all objects admit a deflation from a projective objects. Here projective objects can be defined as those objects such that every deflations into it splits. Instances of exact categories with enough projectives are provided by the resolving subcategories of module categories in the following sense:

\begin{definition}
Let $\mathcal{C}$ be an exact category with enough projectives. A full subcategory $\mathcal{F}\subseteq \mathcal{C}$ is called a \textbf{resolving subcategory} of $\mathcal{C}$ if $\mathcal{F}$ is closed under summands, it is extension-closed and closed under kernels of epimorphisms, and it contains the projective objects of $\mathcal{C}$. 
\end{definition}

Examples of such categories include the previously mentioned categories of Gorenstein projective modules and the category of filtered modules for quasi-hereditary algebras. In fact, Enomoto used the Yoneda embedding to prove that exact categories with enough projectives generally arise as resolving subcategories:

\begin{proposition}[{\cite[Proposition 2.8]{Eno17}}]
Let $\mathcal{C}$ be an be an essentially small idempotent-complete exact category with enough projectives $\mathcal{P}$. Then $\mathcal{C}$ is a resolving subcategory of the category $\modu_\infty \mathcal{P}$ of $\mathcal{P}$-modules with an infinite resolution by representable functors. 
\end{proposition}

Note that in general $\modu_\infty \mathcal{P}$ is not an abelian category. However, under the mild assumption that $\mathcal{P}$ has weak kernels, it is and it coincides with the category of finitely presented $\mathcal{P}$-modules. 

The most important example of a resolving subcategory for our purposes is the category of modules with a standard filtration for a quasi-hereditary algebra. To define such an algebra $A$, one has to fix a partial order $\leq$ on the set of isomorphism classes of simple $A$-modules. With respect to this partial order, one defines the $\mathtt{i}$-th \textbf{standard module} as the maximal highest weight quotient module of an indecomposable projective

\[\Delta_\mathtt{i}=P_\mathtt{i}/\sum_{\substack{f\colon P_\mathtt{j}\to P_\mathtt{i}\\\mathtt{j}\nleq i}} \image f.\]

An algebra is then called \textbf{quasi-hereditary} if $\Delta_\mathtt{i}$ is a skew field and the algebra $A$ belongs to the full subcategory 

\[\mathcal{F}(\Delta)=\{M\in \modu A\,|\,\exists 0=M_0\subset M_1\subset \dots\subset M_t, M_j/M_{j-1}\cong \Delta_{\mathtt{i}_j}\]
of modules admitting a filtration by standard modules. 

Another example of a resolving subcategory of a module category is the monomorphism category, that is the category of modules over $A\otimes_\Bbbk \Bbbk Q$ for a finite acyclic quiver $Q$ and a finite-dimensional algebra $A$. It consists of those modules for which the map
\[M_{\mathtt{i}, \operatorname{in}}\colon \bigoplus_{\substack{\alpha\\t(\alpha)=\mathtt{i}}} M_{s(\alpha)}\xrightarrow{(M_\alpha)} M_\mathtt{i}\]
is a monomorphism for all $\mathtt{i}$. We invite the reader to have a look at Kvamme's survey in these proceedings. As a yet unpublished part of our ongoing joint work with Gao, Kvamme, and Psaroudakis, which so far lead to the two articles \cite{GKKP22, GKKP23}, we proved that like the category of filtered modules over a quasi-hereditary algebra, this category can be described as a filtration closure.

\begin{proposition}
Let $A$ be a finite-dimensional algebra and let $Q$ be a finite acyclic quiver. Then the monomorphism category is equal to the extension closure of the modules $(A\otimes_\Bbbk \Bbbk Q) \otimes_{A\otimes_\Bbbk \Bbbk Q_0} L_\mathtt{i}$, where $L$ runs through the simple $A$-modules and $L_\mathtt{i}$ denotes the simple $A\otimes_\Bbbk \Bbbk Q_0$-module concentrated at the vertex $\mathtt{i}\in Q_0$.  
\end{proposition}  

In general, an old result due to Auslander and Reiten shows that many resolving subcategories arise in this way.

\begin{theorem}[{\cite[Proposition 3.8]{AR91}}]
Let $A$ be a finite-dimensional algebra and let $\mathcal{F}$ be a contravariantly finite resolving subcategory of $\modu A$. Then $\mathcal{F}$ is equal to the idempotent-closure of the extension-closure of the right approximations of the simple $A$-modules. 
\end{theorem}

Note that even if a resolving subcategory is of the form $\mathcal{F}(\Delta)$ for some collection of modules $\Delta$, this does not mean that these modules are the right approximations of the simple modules. 

\begin{example}
Let $A=\Bbbk(\mathtt{1}\to \mathtt{2}\leftarrow \mathtt{3})$. Then the resolving subcategory 
\[\mathcal{F}=\{M\,|\,L(\mathtt{3})\text{ is not a direct summand of } M\}\] 
is equal to $\mathcal{F}(\Delta)$ where the partial order on $\{\mathtt{1},\mathtt{2},\mathtt{3}\}$ is the natural order. However, the right approximation of $L(\mathtt{3})$ is the indecomposable injective $I(\mathtt{2})$ which is not isomorphic to the standard module corresponding to $\mathtt{3}$, which instead is the indecomposable projective  module $P(\mathtt{3})$. 
\end{example}

Nevertheless, the standard modules $\Delta_\mathtt{i}$ for a quasi-hereditary algebra can be characterised categorically inside $\mathcal{F}(\Delta)$. They are those objects such that for every non-zero proper inclusion $M\hookrightarrow \Delta_\mathtt{i}$, where $M\in \mathcal{F}(\Delta)$, the quotient lies outside $\mathcal{F}(\Delta)$. More generally, the notion of a simple object in an exact category has been studied recently by different authors:    

\begin{definition}[{\cite[Definition 3.3]{BHLR20}, \cite[Definition 2.6]{Eno22}}]
Let $\mathcal{E}$ be an exact category. An object $X\in \mathcal{E}$ is called \textbf{simple} (with respect to the exact structure) if for every inflation $Y\rightarrowtail X$ either $Y=0$ or $Y\cong X$.  
\end{definition}

However, this notion of simplicity is not as nicely behaved as for abelian categories as the following example shows.

\begin{example}
Let $A$ be the path algebra of the quiver $\mathtt{1}\to \mathtt{2}\leftarrow \mathtt{3}$. Let $\mathcal{E}$ be the exact subcategory $\add(A\oplus I(\mathtt{2}))$ with the natural exact structure. Then an object in $\mathcal{E}$ is simple if and only if it is indecomposable. However, the exact structure is not split and the simple objects $P(\mathtt{1})$ and $P(\mathtt{3})$ are in the idempotent-closure of the other two simple objects. 
\end{example}

\begin{definition}[{\cite[(2.2)]{Eno22}}]
Let $\mathcal{C}$ be an exact category.
\begin{itemize}
\item An \textbf{inflation series} of $X\in \mathcal{C}$ is a sequence $0=X_0\rightarrowtail X_1\rightarrowtail X_2\rightarrowtail \dots X_t=X$ such that none of the maps are isomorphisms. The number $t$ is called its \textbf{length}.
\item An inflation series is called a \textbf{composition series} if all of its subquotients $X_{j+1}/X_j$ are simple (relative to the exact structure). 
\item The category $\mathcal{C}$ is called a \textbf{length exact category} if there is an upper bound on the lengths of inflation series in $\mathcal{C}$. 
\item A length exact category is said to satisfy the \textbf{Jordan--H\"older property} if for all $X$, all composition series of $X$ have the same length and isomorphic composition factors. 
\end{itemize}
\end{definition}

If $\mathcal{C}$ is the category of finite-dimensional modules for a finite-dimensional algebra with the natural exact structure, then the simple objects in $\mathcal{C}$ are the simple modules and $\mathcal{C}$ is a length exact category satisfying the Jordan--H\"older property by the Jordan--H\"older theorem. Similarly, if $\mathcal{C}$ is the subcategory of finitely generated projective modules with its split exact structure, then the simple objects are the indecomposable projectives and $\mathcal{C}$ is a length exact category satisfying the Jordan--H\"older property by the Krull--Remak--Schmidt theorem for projective modules. It is also well-known that the exact category $\mathcal{F}(\Delta)$ for a quasi-hereditary algebra has as its simple objects the standard modules $\Delta_\mathtt{i}$ and the subcategory satisfies the Jordan--H\"older property. One way to see this is to use that the multiplicities of a $\Delta_\mathtt{i}$ appearing in a $\Delta$-filtration of a module $M$ is determined by the dimension of the space $\Hom(M,\nabla_\mathtt{i})$. 

\begin{question}
What are other examples of exact length categories satisfying the Jordan--H\"older property?
\end{question}

In general, it seems hard to determine when an exact category satisfies the Jordan--H\"older property. For certain subcategories of module categories, there is however a simple numerical criterion due to Enomoto. 

\begin{theorem}[{\cite[Theorem 5.10]{Eno22}}]
Let $\Lambda$ be an Artin algebra and let $U$ be a module of finite injective dimension. Then, 
\[{}^{\perp_{>0}} U=\{X\in \modu \Lambda\,|\,\Ext^i_\Lambda(X,U)=0 \text{ for all } i>0\}.\]
is a length exact category satisfying the Jordan--H\"older property if and only if the number of simples in the exact category ${}^{\perp_{>0}} U$ is equal to the number of indecomposable projectives in it. 
\end{theorem}

A particular instance of such exact categories are given by the contravariantly finite resolving subcategories mentioned before. 

\begin{theorem}[{\cite[Theorem 5.5, Corollary 5.6]{AR91}}]
Let $\Lambda$ be an Artin algebra. Then the contravariantly finite resolving subcategories $\mathcal{F}$ of $\modu \Lambda$ which in addition satisfy that every $\Lambda$-module admits a finite resolution by objects in $\mathcal{F}$ are precisely the subcategories of the form ${}^{\perp_{>0}} T$ for a cotilting module $T$. The additional assumption of admitting finite resolutions is redundant if $\Lambda$ is of finite global dimension. 
\end{theorem}

Using his criterion, Enomoto was able to prove that torsion-free classes for Nakayama algebras satisfy the Jordan--H\"older property, see \cite[Corollary 5.19]{Eno22} as well as give a combinatorial description of when torsion-free classes for path algebras of Dynkin quivers and their preprojective algebras satisfy the Jordan--H\"older property, see \cite[Theorem C, Corollary 5.6]{Eno21}.

Returning to the subject of the previous section, the $A_\infty$-language gives us additional insights in how to define the bound quiver of an exact category. Therefore, we continue by describing a way to construct an exact category from a set of objects filtering every object, the category of twisted modules or twisted stalks. The idea is simple, but requires to technical language of an enhancement to be executed. Since every object can be build iteratively using a sequence of extensions, i.e. elements of $\Ext^1$, the category of twisted modules is defined as said collection of $\Ext^1$'s. However, as there might be obstructions to building an object from a sequence of $Ext^1$'s. To take this into account, there is a condition to ensure that we do not run into any relation, i.e. an element of $\Ext^2$, namely the Maurer--Cartan equation. Given an $A_\infty$-algebra $\mathcal{E}$ (such as the $\Ext$-algebra of some collection of modules), an object of the homotopy category of \textbf{twisted modules} is a strictly upper triangular matrix $x$ in the degree one part $\mathcal{E}^1$ (of any size) such that the Maurer--Cartan equation
\[\sum_{t=1}^\infty (-1)^{\frac{t(t-1)}{2}} m_t(x^{\otimes t})=0\]
is satisfied, where $m_t$ denotes the natural extension of $m_t$ to matrices. Note that the result of the Maurer--Cartan equation lies in the degree two part $\mathcal{E}^2$, so its vanishing means that the sequence of extensions is unobstructed. Morphisms between two matrices $x$, of size $n$, and $y$, of size $m$, are defined to be $m\times n$-matrices $f$ in $\mathcal{E}^0$ such that the equation
\[\sum_{t=0}^\infty\sum_{\substack{i_0,i_1\geq 0\\i_0+i_1=t}} \pm m_{1+t}(y^{\otimes i_1}\otimes f\otimes x^{\otimes i_0})=0\]
holds, where we the reader is referred to \cite[(7.6)]{K01} for the sign. Composition of $f\colon x\to y$ and $g\colon y\to z$ is defined similarly as
\[\sum_{t=0}^\infty\sum_{\substack{i_0,i_1,i_2\geq 0\\i_0+i_1+i_2=t}} \pm m_{2+t}(z^{\otimes i_2}\otimes g\otimes y^{\otimes i_1}\otimes f\otimes x^{\otimes i_0}).\] 

This construction enables one to construct an exact category from a collection of modules filtering it.

\begin{theorem}[{\cite[(7.7)]{K01}}\footnote{The stated reference has the extra assumption that the $A_\infty$-structure may be chosen to be strictly unital. By now it is well-known that this is always the case. To see this, one can e.g. use that every homologically unital $A_\infty$-algebra is $A_\infty$-isomorphic to a strict one, see \cite[Lemma 2.1]{Sei08}}]
Let $\Lambda$ be a finite-dimensional algebra and let $M_\mathtt{1},\dots,M_\mathtt{n}$ be a collection of modules. then there is an equivalence of categories 
\[\mathcal{F}(M)\cong \twmod \Ext^*_\Lambda(M,M).\]
\end{theorem}  

Looking at the degrees of all the elements involved, one sees that one in fact does not need the whole information of the $A_\infty$-structure. It suffices to know the parts with inputs in degrees $0$ and $1$ (and outputs in degrees $0$, $1$, and $2$). There is also a construction of the category of twisted complexes, where one also formally closes under shifts and for which the full $A_\infty$-structure is necessary. Another indictation that the low degrees are especially is the following $A_\infty$-generation property. The result is due to unpublished work of Keller, for a proof see \cite[Theorem 3.22]{KM23b}.

\begin{theorem}
Let $\Lambda$ be a finite-dimensional algebra and let $M_\mathtt{1},\dots,M_\mathtt{n}$ be modules such that $\mathcal{F}(M)$ is resolving. Then $\Ext^{\geq 1}_\Lambda(M,M)$ is generated by $\Ext^1_\Lambda(M,M)$ as an $A_\infty$-algebra. 
\end{theorem}

We conclude the section with a general question about $A_\infty$-structures on $\Ext$-algebras.

\begin{question}
What special properties does the $A_\infty$-algebra $\Ext^*(M,M)$ have when $M$ is the direct sum of the simple objects in some exact category? What are other special properties if the exact category arises as a resolving subcategory of the module category of a finite-dimensional algebra?  
\end{question}

\section{Regular exact subalgebras}\label{regular-exact-subalgebras}

In the representation theory of finite groups and Lie algebras, an important technique is to study certain unital subalgebras and induction and restriction functors. Examples are the Cartan and Borel subalgebras in Lie theory as well as the group algebras of defect groups of blocks of finite groups. This technique is less prevalent in the representation theory of quivers and finite-dimensional algebras. The purpose of this section is to propose a class of subalgebras of finite-dimensional algebras which is well-adapted to studying a particular exact subcategory of a module category. More precisely, the idea is that given a collection of modules $M_\mathtt{1},\dots,M_\mathtt{n}$, e.g. the collection of simples in some exact subcategory, we want to describe $\mathcal{F}(M)$ (or potentially its idempotent closure) via a ring extension $A\subseteq R$ such that $A$ has $\mathtt{n}$ simple modules and $M_\mathtt{i}\cong R\otimes_A M_\mathtt{i}$ for all $\mathtt{i}=\mathtt{1},\dots,\mathtt{n}$. In this section, we want to explore different conditions which make the ring extension more well-behaved, so that it can serve as an analogue of the quiver and relations for $\mathcal{F}(M)$. A first observation is that we can assume $A$ to be basic. A categorical way to see this is to recall that a ring extension $A\subseteq R$ corresponds to the monad $R\otimes_A -$ on the category of $A$-$A$-bimodules. Replacing $A$ up to Morita equivalence by its basic representative $A'$ thus yields a monad on the category $A'$-$A'$-bimodules, which in turn corresponds to a ring extension $A'\subseteq R'$. The remaining properties of the ring extension $A\subseteq R$ we want to demand, can be motivated by looking into the dual coring. This is the perspective taken in \cite{BKK20}. Recall that given a $\Bbbk$-coalgebra $C$, its dual $C^*=\Hom_\Bbbk(C,\Bbbk)$ admits the structure of a $\Bbbk$-algebra. However, conversely given a $\Bbbk$-algebra $R$, one needs that $R$ is finite-dimensional in order to conclude that its dual $R^*$ has the structure of a $\Bbbk$-coalgebra. A similar statement holds when replacing the base field $\Bbbk$ by a (potentially non-commutative) base algebra $A$.

\begin{definition}
Let $A$ be a $\Bbbk$-algebra. An \textbf{$A$-coring} is a comonoid object in the category of $A$-$A$-bimodules, i.e. it is an $A$-$A$-bimodule $V$ together with $A$-$A$-bilinear maps $\mu\colon V\to V\otimes_A V$ and $\varepsilon\colon V\to A$ such that the usual coassociativity and counitality axioms hold. 
\end{definition}

Because $A$ is in general non-commutative, there are two associated duals, already introduced by Sweedler in \cite[Section 3]{Swe75} and studied from a representation-theoretic perspective by Burt and Butler in the proceedings of ICRA 1990 in Tsukuba in \cite{BB91}. The latter reference uses the terminology of left and right algebra. 

\begin{definition}
Given an $A$-coring $V$, its \textbf{left algebra} is the vector space $\Hom_{A^{\op}}(V,A)$ with multiplication induced by $\Hom_{A^{\op}}(\mu,A)$ and unit induced by $\Hom_{A^{\op}}(\varepsilon,A)$. Dually its \textbf{right algebra} is the vector space $\Hom_A(V,A)$ with multiplication induced by the opposite of $\Hom_A(\mu,A)$ and unit induced by $\Hom_A(\varepsilon,A)$. 
\end{definition}

The terminology of left and right algebra (although the definition of the left algebra uses right modules while the definition of the right algebra uses left modules) comes from the fact that with this definition $V$ becomes an $L$-$R$-bimodule. 

To construct an $A$-coring from a ring extension $A\subseteq R$, one needs $R$ to be finitely generated and projective on the appropriate side. This is also naturally from our perspective as it implies that the induction functor from $A$-modules to $\mathcal{F}(M)$ is an exact functor.  Returning to the case of $\Bbbk$-coalgebras, the counit $\varepsilon\colon R\to \Bbbk$ is automatically surjective. This is in general not true for $A$-corings. As it turns out, using Morita theory, one can show that under the hypothesis of $R$ being finitely generated and projective over $A^{\op}$, this is equivalent to $R$ being a projective generator in the category of right $A$-modules, see \cite[Theorem 2.5]{BKK20}. 

Projectivity of $R$ on the right implies that the induction functor $R\otimes_A -\colon \modu A\to \modu R$ induces maps 
\begin{equation}\label{ext}
\Ext^m_A(Y,Z)\to \Ext^m_R(R\otimes_A Y, R\otimes Z).
\end{equation}

\begin{definition}
Let $A\subseteq R$ be a ring extension such that $R$ is finitely generated projective on the right. Then the extension is called 
\begin{itemize}
\item \textbf{homological} if the maps in \eqref{ext} are isomorphisms for $m\geq 2$ and epimorphisms for $m=1$ for all $A$-modules $Y$ and $Z$. 
\item \textbf{regular} if the maps in \eqref{ext} are isomorphisms for all $m\geq 1$ and all \emph{simple} $A$-modules $Y$ and $Z$. 
\end{itemize}
\end{definition}

The terminology `homological' is inspired by the somewhat analogous property of being a homological epimorphism. One might wonder why in the definition for homological there is a condition on all modules while for regularity one uses only simple modules. Using a long exact sequence, to check homologicalness it suffices to determine the condition for simple modules $Y$ and $Z$. On the other hand, from regularity the condition of inducing an isomorphism for all modules is much stronger. Over a perfect ground field, it implies that the inclusion $A\subseteq R$ is a split extension, i.e. there exists an $A$-$A$-bimodule map $\pi\colon R\to A$ with $\pi(1_R)=1_A$, see \cite{Lin23}. On the coalgebra side, this condition is equivalent to the coalgebra being cosplit, which can be checked by the vanishing of certain Hochschild cohomology groups for corings, see \cite[30.16]{BW03}. It would be interesting to more generally investigate different cohomology theories associated to this setup like relative Hochschild cohomology for the ring extension $A\subseteq R$ as well as Cartier and Hochschild cohomology for the coring. 

Coming back to the properties of being homological and regular, these also have a natural description passing to the dual side of corings. The property of being homological is, over a perfect ground field, equivalent to the kernel of the counit being a projective bimodule. More generally, it is equivalent to being projectivising, see \cite{BB91}. 

Assuming an algebraically closed ground field, the property of being regular has on the coring side been studied for a long time, see \cite{KR77}. To explain it better, we need to introduce one more technical condition.   

\begin{definition}
A ring extension $A\subseteq R$ is called \textbf{normal} if the inclusion $A\hookrightarrow R$ admits an algebra splitting whose kernel is a right ideal in $R$. 
\end{definition}

Again, on the algebra side, the condition might not seem too enlightening, however as shown in \cite{Kle84}, for the dual $A$-coring it corresponds to the existence of a group-like element, i.e. an element $\omega\in V$ such that $\mu(\omega)=\omega\otimes \omega$ and $\varepsilon(\omega)=1$. Equivalently, such group-like elements correspond to $V$-comodule structures on $A$ by sending a group-like element to the unique map $A\to V\cong V\otimes_A A$ sending $1$ to $\omega\otimes 1$, see \cite[(28.2)]{BW03}. 

Provided with such an element $\omega\in V$, the condition of regularity translates into the condition that $a\omega - \omega a$ is in the two-sided radical of $\overline{V}=\ker\varepsilon$ for all $a$ in the Jacobson radical of $A$.

We are now ready to propose our analogue of quiver and relations in the context of exact categories. 

\begin{definition}
A \textbf{bound quiver} for an exact subcategory of a module category, given as $\mathcal{F}(M)$ is given by a normal regular ring extension $A\subseteq R$ such that $A$ is basic and $R\otimes_A L_\mathtt{i}\cong M_\mathtt{i}$. 
\end{definition}

Given such a ring extension, the maps \eqref{ext} are not only isomorphisms of vector spaces, but can in fact be shown, see \cite[Theorem 3.21]{KM23b}, to induce an isomorphism of $A_\infty$-algebras 
\[\Ext_A^{\geq 1}(L,L)\to \Ext_R^{\geq 1}(R\otimes_A L, R\otimes_A L).\]

Here it is important to note that because the $m_n$ are of degree $2-n$ the strictly positive part of an $A_\infty$-algebra is closed under the higher multiplications. The existence of the above isomorphism tells us that if we want $R\otimes_A L_\mathtt{i}\cong M_\mathtt{i}$, then from the previous section we deduce that the maps $m_n\colon \Ext^1_R(M,M)\to \Ext^2_R(M,M)$ give us a presentation of the subalgebra $A$ as a quiver with relations. But what about the other higher multiplications, in particular the ones including $\Hom(M,M)$, which need not vanish, even for simple objects in an exact category? There is the following analogue of Lemma \ref{ext-quiver}:

\begin{lemma}
Let $A\subseteq R$ be a basic normal regular algebra extension over an algebraically closed field $\Bbbk$. Denote by $V$ the corresponding dual coring $\Hom_{A^{\op}}(R,A)$. Then, there is are isomorphisms
\begin{itemize}
\item $\mathbb{D}\Hom_R(R\otimes_A L_\mathtt{i},R\otimes_A L_\mathtt{j})\cong e_\mathtt{j}\overline{V}/(\rad(A)\overline{V}+\overline{V}\rad(A))e_\mathtt{i}$,
\item $\mathbb{D}\Ext^1_R(R\otimes_A L_\mathtt{i}, R\otimes_A L_\mathtt{j})\cong e_\mathtt{j}Q_+/Q_+^2e_\mathtt{i}$,
\item $\mathbb{D}\Ext^2_R(R\otimes_A L_\mathtt{i}, R\otimes_A L_\mathtt{j})\cong e_\mathtt{j}I/(Q_+I+IQ_+)e_\mathtt{i}$,
\end{itemize}
where $A\cong \Bbbk Q/I$ for some finite quiver $Q$ and an admissible ideal $I$. 
\end{lemma}

To explain where the higher multiplications come from, we need one more ingredient, namely Roiter's equivalence, which can be seen as a form of relative Koszul duality. 

\begin{theorem}
There is an equivalence of categories between the category of corings with a group-like element and the category of differential graded structures on tensor algebras of degree $1$ bimodules with the natural grading. It is given by sending the pair $(A,V,\mu,\varepsilon,\omega)$ to the tensor algebra $T_A(\overline{V})$ where $\overline{V}=\ker\varepsilon$ and the differential is uniquely induced by 
\begin{align*}
\partial_0(a)&=a\omega-\omega a&\text{for $a\in A$,}\\
\partial_1(v)&=\mu(v) - v\otimes \omega - \omega\otimes v&\text{for $v\in \overline{V}$.} 
\end{align*}
\end{theorem}

The theorem can be generalised to corings with surjective counit where the group-like element $\omega$ is replaced by the choice of a base point with $\varepsilon(\omega)=1$. However, in this case, as is common in non-augmented settings, one has to replace differential graded structures by curved differential graded structures, for details see \cite{Brz13}. 

By the results presented in this section, the question to define a ring extension $A\subseteq R$ is via duality and Roiter's equivalence reduced to constructing a differential graded algebra structure on the tensor algebra over a bimodule. In the next section we discuss how to construct this differential graded algebra.

\section{The case of quasi-hereditary algebras}\label{qh-case}

In this section we turn our attention again to quasi-hereditary algebras, which is the motivating instance of the proposed theory of basic regular exact subalgebras. For one of the simplest examples of a quasi-hereditary algebra, let $A=\Bbbk Q/I$ be the quotient of the path algebra of an acyclic quiver by an admissible ideal $I$. Because the quiver $Q$ is acyclic, one can define a partial order on the vertex set $Q_0$, which also indexes the isomorphism classes of simple modules, by $\mathtt{i}\leq \mathtt{j}$ if there is a path from $\mathtt{i}$ to $\mathtt{j}$ in $Q$. One readily observes that with respect to this partial order, the standard modules are actually simple, hence the algebra is filtered. It follows that $(A,\leq)$ is a special kind of quasi-hereditary algebra which is called a \textbf{directed algebra}. Without mentioning of a partial order, this notion is better known under the name of a triangular algebra. It is not to be confused with the notion of a representation-directed algebra, which is equivalent to the algebra being representation-finite and its Auslander algebra being triangular. More generally, let $A\subseteq R$ be a basic regular exact subalgebra with respect to the standard modules $\Delta_\mathtt{i}$. Because $R$ is quasi-hereditary, the standard modules $\Delta$ form an exceptional collection, in particular $\Ext^1_R(\Delta_\mathtt{i},\Delta_\mathtt{j})\neq 0$ implies $i\leq j$. Because of the isomorphism condition in the definition of a regular exact subalgebra, 
\[\Ext^1_A(L_\mathtt{i}, L_\mathtt{j})\cong \Ext^1_R(\Delta_\mathtt{i}, \Delta_\mathtt{j})\]
also the simple $A$-modules form an exceptional collection. Since extensions between simple modules correspond to arrows in the quiver of $A$, it follows that the Gabriel quiver of $A$ is acyclic with arrows between the module only going in increasing direction, i.e. the algebra $A$ is directed. As a consequence, a regular exact subalgebra with respect to the standard modules is in particular an exact Borel subalgebra in the sense of \cite{Koe95}.

\begin{definition}
Let $R$ be a quasi-hereditary algebra. A subalgebra $A\subseteq R$ is called an \textbf{exact Borel subalgebra} if $A$ has the same number of simple modules as $R$ and the following three conditions are satisfied:
\begin{itemize}
\item $A$ is directed,
\item $R$ is projective as a right $A$-module,
\item $R\otimes_A L_\mathtt{i}^A \cong \Delta_\mathtt{i}^R$.
\end{itemize}
\end{definition} 

The careful reader will have noticed that there are now potentially two different partial orders on an exact Borel subalgebra of a quasi-hereditary algebra. One is induced from the partial order on the quasi-hereditary algebra $R$ and one comes from noticing that the Gabriel quiver of $A$ is acyclic and defining the partial order via existence of paths as above. However, both partial orders yield the same standard modules, namely the simples. This is a general phenomenon for quasi-hereditary algebras and it can therefore be useful to study quasi-hereditary structures on an algebra up to the equivalence relation that they define the same standard modules. It is sometimes convenient to work with particularly well-behaved orders in an equivalence class. A common one are the total orders in an equivalence class obtained by arbitrarily refining any partial order in the equivalence class. Another one is the essential order which is the coarsest partial order in an equivalence class. 

\begin{definition}[{\cite[Definition 1.2.5]{Cou20}}]
Let $(R,\leq)$ be a quasi-hereditary algebra. Then the \textbf{essential partial order} for $R$ is the partial order induced by $\mathtt{i}<_{\ess}\mathtt{j}$ if $[\Delta_\mathtt{j}:L_\mathtt{i}]\neq 0$ or $(P_\mathtt{i}:\Delta_\mathtt{j})\neq 0$, where the latter denotes the number of times $\Delta_\mathtt{j}$ appears as a subquotient in any $\Delta$-filtration of $P_\mathtt{i}$. 
\end{definition} 

A result by Mazorchuk, which we include with his permission here, is that the essential order can equivalently be understood in terms of the $\Ext$-algebra of the standard modules.

\begin{lemma}\label{walter's-lemma}
The essential partial order for a quasi-hereditary algebra $R$ can equivalently be defined as the partial order induced by $\mathtt{i}<_{\ess} \mathtt{j}$ if $\Hom_R(\Delta_\mathtt{i},\Delta_\mathtt{j})\neq 0$ or $\Ext^1_R(\Delta_\mathtt{i},\Delta_\mathtt{j})\neq 0$. 
\end{lemma}

\begin{proof}
We first want to show that if $\Hom_R(\Delta_\mathtt{i},\Delta_\mathtt{j})\neq 0$ or $\Ext^1_R(\Delta_\mathtt{i},\Delta_\mathtt{j})\neq 0$, then $\mathtt{i}<_{\ess} \mathtt{j}$. If $\Hom_R(\Delta_\mathtt{i},\Delta_\mathtt{j})\neq 0$, then composition of a non-zero map in this space with the projection $P_\mathtt{i}\twoheadrightarrow \Delta_\mathtt{i}$ yields a non-zero map $P_\mathtt{i}\to \Delta_\mathtt{j}$, i.e. that $[\Delta_\mathtt{j}:L_\mathtt{i}]\neq 0$. On the other hand, assuming that $\Ext^1_R(\Delta_\mathtt{i},\Delta_\mathtt{j})\neq 0$, let $K_\mathtt{i}$ be the kernel of the projection map $P_\mathtt{i}\twoheadrightarrow \Delta_\mathtt{i}$. The assumed non-vanishing implies that $\Hom_R(K_\mathtt{i},\Delta_\mathtt{j})\neq 0$. This means that there exists $\mathtt{k}$ such that $\Delta_\mathtt{k}$ is a subquotient of $K_\mathtt{i}$ and $\Hom_R(\Delta_\mathtt{k},\Delta_\mathtt{j})\neq 0$, which by the earlier case implies $(P_\mathtt{i}:\Delta_\mathtt{k})\neq 0$ and $[\Delta_\mathtt{j}:L_\mathtt{k}]\neq 0$. It follows that $\mathtt{i}\leq_{\ess} \mathtt{k}$ and $\mathtt{k}\leq_{\ess} \mathtt{j}$, whence $\mathtt{i}\leq_{\ess} \mathtt{j}$. 
Conversely, we have to show that if $[\Delta_\mathtt{j}:L_\mathtt{i}]\neq 0$ or $(P_\mathtt{i}:\Delta_\mathtt{j})\neq 0$, then there exists a sequence of non-zero homomorphisms and extensions between standard modules from $\Delta_\mathtt{i}$ to $\Delta_\mathtt{j}$. Suppose first that $(P_\mathtt{i}:\Delta_\mathtt{j})\neq 0$, this means that $\Delta_\mathtt{j}$ appears as a subfactor in a $\Delta$-filtration of $P_\mathtt{i}$. Since any $\Delta$-filtration of $P_\mathtt{i}$ has to have $\Delta_\mathtt{i}$ as a top filtration factor, this sequence comes with a non-vanishing sequence of extensions between standard modules from $\Delta_\mathtt{i}$ to $\Delta_\mathtt{j}$ (cf. the description of the category of $\Delta$-filtered modules via twisted modules). On the other hand, assuming $[\Delta_\mathtt{j}:L_\mathtt{i}]=\dim\Hom_R(P_\mathtt{i},\Delta_\mathtt{j})\neq 0$, by a long exact sequence argument implies that there exists $\mathtt{k}$ with $(P_\mathtt{i}:\Delta_\mathtt{k})\neq 0$ and $\Hom_R(\Delta_\mathtt{k},\Delta_\mathtt{j})\neq 0$, by the preceding case, the first one implies that there is a sequence of non-vanishing extensions between standards from $\Delta_\mathtt{i}$ to $\Delta_\mathtt{k}$. The claim follows. 
\end{proof}

Together with Conde we noted that one can easily define examples where the notion of being an exact Borel subalgebra depends on the chosen partial order within an equivalence class. However, this cannot happen for regular exact Borel subalgebra. 

One of the main results of our joint work with Koenig and Ovsienko was the following existence result for Borel subalgebras.

\begin{theorem}
Let $\Bbbk$ be an algebraically closed field. Then for every quasi-hereditary algebra $\Lambda$ there is a Morita equivalent quasi-hereditary algebra $R$ which has a basic normal regular exact Borel subalgebra.
\end{theorem}

The construction starts with computing the $A_\infty$-structure on $\Ext^*_\Lambda(\Delta,\Delta)$. To produce a differential graded algebra, one first applies the bar construction (over the product of the scalar multples of the identities) to this $A_\infty$-algebra. This produces a differential graded structure on its tensor coalgebra (whose differential is induced by the sum of all the $m_n$, shifted appropriately so that they are of degree $1$). Taking its graded dual, one obtains a differential graded algebra. However, it will usually not be a tensor algebra over its degree zero part, so we can't immediately apply Roiter's equivalence. An extra truncation step is required where one quotients by the differential ideal generated by all the negative degree parts. This indeed yields a tensor algebra over the degree zero part, to which we apply Roiter's equivalence to obtain a coring whose right algebra is the Morita equivalent algebra $R$.  

As mentioned before, the low degree parts of the $\Ext$-algebra are sufficient in particular to construct the exact category $\mathcal{F}(\Delta)$ as a category of twisted modules. It should therefore not come as a surprise that these parts are also sufficient to reconstruct the extension $A\subseteq R$ (or equivalently the differential graded structure on $T_A(\overline{V})$). As in the case of simple modules, the sum of the duals of the $A_\infty$-multiplications
\[m_n\colon (\Ext^1_\Lambda(\Delta,\Delta))^{\otimes n}\to \Ext^2_\Lambda(\Delta,\Delta)\]
provides a presentation of the algebra $A$ as a quotient of the tensor algebra (over a product of fields) of $\mathbb{D}\Ext^1_\Lambda(\Delta,\Delta)$. In contrast to the case of simple modules one has to take care of the (radical) homomorphisms between standard modules as well. In fact, $\overline{V}$ is the projective $A$-$A$-bimodule generated by  the radical maps $\rad_\Lambda(\Delta,\Delta)$ and the differential on $T_A(\overline{V})$ is determined by two maps, $\partial_0\colon A\to \overline{V}$ and $\partial_1\colon \overline{V}\to \overline{V}\otimes \overline{V}$. The former is induced from the duals of the higher multiplications 
\[m_n\colon \bigoplus_{i+j=n-1} (\Ext^1_\Lambda(\Delta,\Delta))^{\otimes i}\otimes \Hom_\Lambda(\Delta,\Delta)\otimes (\Ext^1_\Lambda(\Delta,\Delta))^{\otimes j}\to \Ext^1_\Lambda(\Delta,\Delta),\]
the latter by the duals of 
\[m_n\colon \bigoplus_{i+j+k=n-2} (\Ext^1_\Lambda)^{\otimes i}\otimes \Hom_\Lambda\otimes (\Ext^1_\Lambda)^{\otimes j}\otimes \Hom\otimes (\Ext^1)^{\otimes k}\to \Hom_\Lambda,\]
where as before all homomorphisms and extensions are between $\Delta$. 

If one wants to more generally construct a regular exact subalgebra for a collection of modules, one would expect that the same construction works. However, it is easy to construct counterexamples where the resulting $A$ is not finite-dimensional, like the following:

\begin{example}
Let $\Lambda=\Bbbk(\begin{tikzcd}\mathtt{1}\arrow[yshift=0.5ex]{r}\arrow[yshift=-0.5ex]{r}&\mathtt{2}\end{tikzcd})$ be the path algebra of the Kronecker quiver. Let $M_\mathtt{1}$ be the simple projective module and let $M_\mathtt{2}$ be a quasi-simple module. Then, the $\Ext^1$-quiver between the $M_\mathtt{i}$ looks as follows
\[
\begin{tikzcd}
M_\mathtt{1}&M_\mathtt{2}\arrow[loop right]{r}\arrow{l}
\end{tikzcd}
\]
Since $\Ext^2(M,M)=0$, the resulting $A$ would be the path algebra of this non-acyclic quiver, in particular infinite-dimensional. 
\end{example}

The reason why the construction works in the quasi-hereditary case, and is relatively easy to understand is that the $\Delta_\mathtt{i}$ form an exceptional collection and therefore even the bar construction of the $\Ext$-algebra of the standard modules (relative to the product of scalar multiples of the identities) is finite-dimensional. My guess is that the fact that $A$ is infinite-dimensional in the example is related to the non-functorially finiteness of $\mathcal{F}(M)$. It would be interesting to have a general answer to the question of when it is possible to construct a regular exact subalgebra for a collection of modules. For this, one would in particular need to answer the following question.

\begin{question}
Let $\mathcal{C}$ be an exact (length) category. Under which conditions is the $A_\infty$-Koszul dual of $\Ext^{\geq 1}_{\mathcal{C}}(M,M)$, where $M$ denotes the direct sum of the simple objects in $\mathcal{C}$, finite-dimensional? 
\end{question}

Returning to the situation for quasi-hereditary algebras, in joint work with Vanessa Miemietz, we were able to also prove uniqueness of basic regular exact Borel subalgebras up to isomorphism. Our proof is quite involved, the basic idea is that given choices of presentations for $A$ and $V$, to construct an $A_\infty$-structure on $\Ext^*_\Lambda(\Delta,\Delta)$ such that when applying the construction yielding existence to it, one gets back $A$ and $V$ (with the chosen presentations). Once this is done, one can invoke Kadeishvili's uniqueness theorem for the $A_\infty$-structure on the $\Ext^*_\Lambda(\Delta,\Delta)$. 

\begin{theorem}
Let $\Bbbk$ be an algebraically closed field. Let $A\subseteq R$ be a regular exact Borel subalgebra and let $B\subseteq S$ be a regular exact Borel subalgebra such that $R$ is Morita equivalent to $S$ as a quasi-hereditary algebra. Then there is an isomorphism $f\colon R\to S$ sending $A$ to $B$. 
\end{theorem}

Comparing to the situation of a regular exact subalgebra for the collection of projective modules, i.e. a maximal semisimple subalgebra, the Wedderburn--Malcev theorem shows uniqueness not only up to isomorphism, but when considering two maximal semisimple subalgebras of the same algebra, uniqueness up to conjugation. We expect such a result to hold more generally in our setup.

\begin{conjecture}
Let $\Bbbk$ be an algebraically closed field. Let $A,B\subseteq R$ be basic regular exact Borel subalgebras. Then there exists an invertible element $x\in R$ such that $x^{-1}Ax=B$. 
\end{conjecture}

We suggested to call a basic regular exact Borel subalgebra $A\subseteq R$ a bound quiver for the exact category $\mathcal{F}(\Delta)$. Though a big difference to the classical theory of bound quivers is that here there are two algebras involved and we only ask $A$ to be basic. The algebra $R$ is usually quite far from being basic. In fact, Conde has provided the following interesting characterisation of the quasi-hereditary algebras such that $R$ is basic. 

\begin{theorem}[{\cite[Proposition 5.1]{Con21}}]
Let $R$ be a basic quasi-hereditary algebra. Then $R$ has a regular exact Borel subalgebra if and only if $\rad \Delta_\mathtt{i}\in \mathcal{F}(\nabla)$ for all $\mathtt{i}$.  
\end{theorem}

More generally, she was even able to give a recursive formula for the decomposition multiplicies of the projectives appearing in $R$, which is quite helpful and easy to use in practice if one wants to know which Morita representative admits a regular exact Borel without computing the whole $A_\infty$-structure on the $\Ext$-algebra first. 

\begin{theorem}\label{Conde's-formula}
Let $\Lambda$ be a quasi-hereditary algebra. Define a sequence of positive integers $\ell_\mathtt{i}$ recursively, via 
\[\ell_\mathtt{i}=1+\sum_{\mathtt{k}\leq \mathtt{j}<\mathtt{i}} \ell_\mathtt{k} [\nabla_\mathtt{j}:L_\mathtt{k}] \dim \Hom_\Lambda(\Delta_\mathtt{j}, \Delta_\mathtt{i})-\sum_{\mathtt{j}<\mathtt{i}} \ell_j[\Delta_\mathtt{i}:L_\mathtt{j}]\]
Then, $R\cong \End(P_\mathtt{i}^{\ell_\mathtt{i}})^{\op}$ is the unique Morita representative of $\Lambda$ having a basic regular exact Borel subalgebra. 
\end{theorem}

Note that always $\ell_\mathtt{i}=1$ if $\mathtt{i}$ is minimal with respect to the partial order.  

\section{Examples}\label{examples}

When existence of regular exact Borel subalgebras was proven, there weren't too many examples of quasi-hereditary algebras were the in principle constructive approach was done. In fact, the appendix of \cite{KKO14} contains just $4$ examples. A few more examples, including the representation-finite and tame blocks of Schur algebras are contained in the survey \cite{Kul17}. In the past years, the situation has changed a bit and theoretical results as well as computations for particular classes of quasi-hereditary algebras have been put forward. 

We start this section with the class of path algebras of finite quivers. As hereditary algebras they are quasi-hereditary with respect to every adapted order, see \cite{DR92} (for a discussion of different ways to define the concept of an adapted order see \cite{Rod23}). However, many adapted orders will yield the same standard modules. In \cite{FKR22}, Flores, Kimura, and Rognerud classified quasi-hereditary structures on several classes of such algebras. In particular, they proved that the number of quasi-hereditary structures on the path algebra of a linearly oriented $\mathbb{A}_n$ quiver is given by the $n$-th Catalan numbers and that in general quasi-hereditary structures can be glued at sinks or sources. As a consequence, Thuresson obtain a general criterion for the path algebra of a quiver of Dynkin type $\mathbb{A}$ to admit a regular exact Borel subalgebra. 

\begin{theorem}[{\cite[Theorem 59]{Thu24}}]
Let $A$ be the path algebra of an $\mathbb{A}_n$-quiver $Q$. Then, $A$ has a regular exact Borel subalgebra if and only if every sink with two neighbours is extremal (i.e. minimal or maximal) with respect to the essential order. In particular, every linearly oriented $\mathbb{A}_n$-quiver has an exact Borel subalgebra, independently of the chosen partial order. 
\end{theorem}

Thuresson's paper also gives a combinatorial description of the quiver of the exact Borel subalgebras. However, he doesn't compute the decomposition multiplicities of $R$ in general (i.e. in the case when there is a non-extremal sink). As an illustration of Conde's formula, we compute multiplicities in one particular example. 

\begin{example}
Let $A$ be the path algebra of the $\mathbb{A}_{2n+1}$-quiver with zig-zag orientation 
\[\begin{tikzcd}\mathtt{1}\arrow{r} &\mathtt{2}&\mathtt{3}\arrow{l}\arrow{r}&\dots&\mathtt{2n+1}\arrow{l}\end{tikzcd}\]
with the natural partial order on $\{\mathtt{1},\mathtt{2}, \dots, \mathtt{n}\}$, which is also the essential order for this algebra. Then, the standard modules are given by $\Delta_\mathtt{1}=L_\mathtt{1}$ and $\Delta_\mathtt{2i}=L_\mathtt{2i}$ and $\Delta_\mathtt{2i+1}=M(\mathtt{2i}, \mathtt{2i+1})$, the length two interval module with top $L_\mathtt{2i+1}$ and socle $L_\mathtt{2i}$ for all $\mathtt{i}\geq \mathtt{1}$. It is easily seen that 
\[\dim \Ext^1_A(\Delta_\mathtt{i}, \Delta_\mathtt{j})=\begin{cases}1&\text{if $\mathtt{i}$ is odd and $\mathtt{j}=\mathtt{i+1}$ or $\mathtt{j}=\mathtt{i+2}$},\\0&\text{else.}\end{cases}\]
Thus, the regular exact Borel subalgebra of an algebra $R$ Morita equivalent to $A$ is isomorphic to the path algebra of a quiver with the following shape:
\[
\begin{tikzcd}
\mathtt{1}\arrow{d}\arrow{r}&\mathtt{3}\arrow{r}\arrow{d}&\dots\arrow{r}&\mathtt{2n-1}\arrow{r}\arrow{d}&\mathtt{2n+1}\\
\mathtt{2}&\mathtt{4}&\dots&\mathtt{2n}
\end{tikzcd}\]
By Thuresson's theorem above, $R$ is not isomorphic to $A$ as $\mathtt{2}, \mathtt{4},\dots,\mathtt{2n}$ are not extremal. To use  Conde's formula \ref{Conde's-formula}, note that for $\mathtt{i}$ even, $\Hom(\Delta_\mathtt{j},\Delta_\mathtt{i})=0$ and $[\Delta_\mathtt{i}:L_\mathtt{j}]=0$ for $\mathtt{j}\neq \mathtt{i}$. It follows that $\ell_\mathtt{i}=1$ for even $\mathtt{i}$. In the case that $\mathtt{i}$ is odd, note that $\Hom(\Delta_\mathtt{j},\Delta_\mathtt{i})$ vanishes except for $\mathtt{j}=\mathtt{i-1}$, in which case it is one-dimensional. Similarly, $[\nabla_\mathtt{i-1}:L_\mathtt{k}]$ vanishes except for $\mathtt{k}=\mathtt{i-1}$ and $\mathtt{k}=\mathtt{i-2}$, in which case the multiplicities are one. Finally, $[\Delta_\mathtt{i}:L_\mathtt{j}]=0$ except for $j=\mathtt{i-1}$. Therefore, 
\[\ell_i=1+\sum_{\mathtt{k}=\mathtt{i-1} \text{ or }\mathtt{i-2}} \ell_\mathtt{k}-\sum_{\mathtt{j}=\mathtt{i-1}} \ell_\mathtt{j}=1+\ell_\mathtt{i-2}.\] 
It follows that \[R\cong \End_A(A\oplus 2P_\mathtt{3}\oplus 3P_\mathtt{5} \dots\oplus (n+1)P_\mathtt{2n+1})^{\op},\] so even for type $\mathbb{A}$-quivers, multiplicities of projectives in the algebra can get quite big.   
\end{example}

There is no general description for non-linearly oriented type $\mathbb{A}$ quivers for the decomposition multiplicities of $R$. For the other Dynkin types, and more general hereditary algebras, it would be interesting to obtain a criterion for when the quasi-hereditary algebra has a regular exact Borel subalgebra. We continue by giving the example of a path algebra of a finite quiver which is not of type $\mathbb{A}$.

\begin{example}
Another example for how high-dimensional $R$ can get even for fairly innocent-looking algebra is given by the following version of the linearly ordered $\mathbb{A}_3$-quiver with multiple arrows:
\[
\begin{tikzcd}
\mathtt{1}\arrow[phantom]{r}[description]{\vdots}\arrow[yshift=1ex]{r}{a_1}\arrow[yshift=-1.5ex]{r}[swap]{a_m}&\mathtt{2}\arrow[phantom]{r}[description]{\vdots}\arrow[yshift=1ex]{r}{b_1}\arrow[yshift=-1.5ex]{r}[swap]{b_n}&\mathtt{3}
\end{tikzcd}
\]
According to \cite[Proposition 2.23]{FKR22}, for a hereditary algebra, the number of quasi-hereditary structures does not depend on the multiplicity of the arrows. Therefore, just like for linearly oriented $\mathbb{A}_3$, there are $5$ quasi-hereditary structures on this algebra. All but one of them satisfy the criterion $\rad \Delta_\mathtt{i}\in \mathcal{F}(\nabla)$ since the radical of $\Delta_\mathtt{i}$ either vanishes or is a simple costandard module. The only remaining partial order is $\mathtt{1}>\mathtt{3}>\mathtt{2}$. With respect to this partial order, the standard modules are $\Delta_\mathtt{1}=P_\mathtt{1}$, $\Delta_\mathtt{2}=L_\mathtt{2}$, and $\Delta_\mathtt{3}=L_\mathtt{3}$ while the costandard modules are $\nabla_\mathtt{1}=L_\mathtt{1}$, $\nabla_\mathtt{2}=L_\mathtt{2}$, and $\nabla_\mathtt{3}$ is a module with top $L_\mathtt{2}^m$ and socle $L_\mathtt{3}$. Computing extensions, one can check that the basic regular exact Borel subalgebra is given by
\[
\begin{tikzcd}
\mathtt{3}&\mathtt{2}\arrow[phantom]{l}[description]{\vdots}\arrow[yshift=1ex]{l}[swap]{b_1}\arrow[yshift=-1.5ex]{l}{b_n}\arrow[phantom]{r}[description]{\vdots}\arrow[yshift=1ex]{r}{c_1}\arrow[yshift=-1.5ex]{r}[swap]{c_{(n^2-1)m}}&\mathtt{1}.
\end{tikzcd}
\]
As is noted in \cite{Con21}, the minimal elements and the elements one layer up in the partial order always have multiplicity $1$, thus the only remaining multiplicity is
\begin{align*}
\begin{split}
\ell_\mathtt{1}&=1+[\nabla_\mathtt{3}:L_\mathtt{2}]\Hom(\Delta_\mathtt{3}:\Delta_\mathtt{1})+[\nabla_\mathtt{3}:L_\mathtt{3}]\Hom(\Delta_\mathtt{3},\Delta_\mathtt{1})\\
&\phantom{1}-[\Delta_\mathtt{1}:L_\mathtt{3}]-[\Delta_\mathtt{1}:L_\mathtt{2}]
\end{split}\\
&=1+n\cdot (nm)+1\cdot (nm)-nm-m=1+(n^2-1)m.
\end{align*}
It follows that the basic representative of this quasi-hereditary algebra has a regular exact Borel subalgebra if and only if $n=1$, in general, $R\cong \End_A(A\oplus P_\mathtt{1}^{(n^2-1)m})$ has a basic regular exact Borel subalgebra. 
\end{example}

Another method to construct quasi-hereditary algebras are Xi's dual extension algebras. To construct them one starts with two algebras $B=\Bbbk Q/I$ and $A=\Bbbk Q'/I'$ where the vertex sets of $Q$ and $Q'$ coincide. One then glues $A$ and $B$ in the following way: The \textbf{dual extension algebra} of $A$ and $B$ is the quotient of the path algebra of the joint quiver $\overline{Q}$ with arrows in $Q_1\cup Q_1'$ by the relations generated by $I$, $I'$, and the zero relations $\alpha\beta'$ where $\alpha\in Q$ and $\beta'\in Q'$. If one assumes that $B$ and $A^{\op}$ are directed, then $\Lambda$ is quasi-hereditary with respect to the same partial order. Its standard modules are the projective $A$-modules. Furthermore, the algebra $B$ is an exact Borel subalgebra of $\Lambda$. However, it will usually not be regular. As outlined, an essential step in the construction of a regular exact Borel subalgebra is the computation of the $\Ext$-algebra of the standard modules. In this regard, Thuresson obtained the following result:

\begin{theorem}[{\cite[Theorem 10]{Thu22}}]
Let $\Lambda$ be the dual extension algebra of $A$ and $B$. Then the $\Ext$-algebra of its standard modules is the dual extension algebra of the $\Ext$-algebra of the simple $B$-modules and $A$. 
\end{theorem}

His paper also contains methods to obtain the $A_\infty$-structure on the $\Ext$-algebra of standard modules from the $A_\infty$-structure on the $\Ext$-algebra of simple $B$-modules. The above result has been generalised to Deng and Xi's double twisted incidence algebras by Norlén J\"aderberg, see \cite[Theorem 2.48]{Nor23}. Interestingly, the Ringel duals of the dual extension algebras of path algebras of trees satisfy Conde's criterion and thereby have regular exact Borel subalgebras, see \cite[Proposition 5.5]{Con21}.  

Another common way to construct new classes of finite-dimensional algebras with similar homological properties from a given one is given by skew group algebras: Let $\Lambda$ be an algebra and $G$ be a finite group acting on $\Lambda$ via algebra automorphisms. Then, the skew group algebra $\Lambda*G$ is defined to be the algebra with underlying vector space $\Lambda\otimes_\Bbbk \Bbbk G$ and multiplication defined by 
\[(a\otimes g)(b\otimes h)=ag(b)\otimes gh\]
for all $a,b\in \Lambda$ and all $g,h\in G$ where $g(b)$ denotes the action of $g$ on $b$. In general, there is no reason why the skew-group algebra of a quasi-hereditary algebra should be quasi-hereditary. One has to ask for the following compatibility: A strict partial order $<$ on the isomorphism classes of simple $\Lambda$-modules is called \textbf{$G$-equivariant} if $L_\mathtt{i}<L_\mathtt{j}$ if and only if $L_\mathtt{i}^{(g)}<L_\mathtt{j}^{(h)}$ for all $g,h\in G$. 

\begin{theorem}[{\cite[Theorem 3.13, Proposition 3.16]{Rod23}}]
Let $\Lambda$ be a quasi-hereditary algebra. Let $G$ be a finite group such that the characteristic of the ground field does not divide the order of $G$. Suppose that $G$ acts on $\Lambda$ via algebra automorphisms. Assume that the induced action on the isomorphism classes of simple $\Lambda$-modules is $G$-equivariant. Then the skew group algebra $\Lambda * G$ is quasi-hereditary. If $\Lambda$ has a normal regular exact Borel subalgebra $B$, which is invariant under the action of $G$, then $B*G$ is a normal regular exact Borel subalgebra of $\Lambda * G$. 
\end{theorem}

As group algebras are typically not basic, we cannot expect the algebra $B*G$ in the above theorem to be basic. However, as mentioned before, we can obtain a bound quiver for $\Lambda * G$ by choosing a basic representative for $B*G$. We conclude our survey with an example from Lie theory. 

\begin{example}\label{category-o-example}
Let $\mathfrak{g}$ be a finite-dimensional complex semisimple Lie algebra with a fixed triangular decomposition $\mathfrak{g}=\mathfrak{n}^-\oplus \mathfrak{h}\oplus \mathfrak{n}^+$. Let $\mathcal{O}$ be BGG category, that is the category of finitely generated locally $U(\mathfrak{n}^+)$-finite weight modules. Then, $\mathcal{O}$ decomposes into blocks $\mathcal{O}_\chi$, where each block $\mathcal{O}_\chi$ is equivalent to the module category of a finite-dimensional algebra $A_\chi$. Here we focus on the principal block $\mathcal{O}_0$. The simple objects in $\mathcal{O}$ are indexed by weights of the form $w\cdot (-2\rho)$, where $w\in W$, the corresponding Weyl group and $\cdot$ denotes the dot action on weights defined by $w\cdot \lambda=w(\lambda+\rho)-\rho$, where $\rho$ denotes the half sum of the positive roots. The Bruhat order on $W$ thus gives a partial order on the simple modules belonging to the principal block and thus a partial order on the simples for the finite-dimensional algebra $A_\chi$. With respect to this partial order, the algebra $A_\chi$ is quasi-hereditary. Its standard modules correspond to the Verma modules $\Delta_w=U(\mathfrak{g})\otimes_{U(\mathfrak{b})} \mathbb{C}_{w\cdot (-2\rho)}$. In \cite{Koe95}, Koenig proved that a particular Morita representative of $A_\chi$ has an exact Borel subalgebra. It is not known whether this particular representative has a regular exact Borel subalgebra. We want to apply Conde's formula \ref{Conde's-formula} to investigate the multiplicities of the projective modules in this case. Recall (e.g. from \cite[Theorems 7.6.6, 7.7.7]{Dix96} or \cite[Theorem 4.2, Remark 5.1, Corollary 5.2]{Hum08}) that the Verma modules satisfy 
\[\dim \Hom(\Delta_x,\Delta_w)=\begin{cases}1&\text{if $x\leq w$,}\\0&\text{else.}\end{cases}\]
More generally, this property holds for the version of category $\mathcal{O}$ associated to an arbitrary Coxeter group, see \cite[Theorem 6.1]{Abe12}. 
On the other hand, the Kazhdan--Lusztig conjectures state that the multiplicities $[\Delta_x:L_w]$ (and $[\nabla_x:L_w])$ are given by the value of the corresponding Kazhdan--Lusztig polynomials at $1$, i.e. $P_{w_0x, w_0w}(1)$. Thus, for blocks of category $\mathcal{O}$, Conde's formula \ref{Conde's-formula} specialises to 
\[\ell_x=1+\sum_{v\leq w<x} \ell_\nu P_{w_0w, w_0v}(1)-\sum_{w<x}\ell_w P_{w_0x, w_0w}(1).\]
If the rank of $\mathfrak{g}$ is smaller than or equal to $2$, then Jantzen determined that all non-zero multiplicities are equal to one, see \cite[(8.1)]{Hum08}. Again, the same statement holds more generally for principal blocks of the generalisations of category $\mathcal{O}$ associated to dihedral groups, see \cite[Proof of Theorem 7.2]{Sau18}. We claim that in this case we have the following formula for the multiplicities:
\[\ell_w=\begin{cases}1&\text{if $w$ is minimal,}\\3^{\height(w)-1}&\text{else,}\end{cases}\]
where $\height(w)$ denotes the height of $w$ in the poset $W$ equipped with the Bruhat order. To see this, we start by using the specialisation of Conde's formula to the case where $\dim\Hom(\Delta_w,\Delta_x)=1$, $[\Delta_w:L_x]=1$, and $[\nabla_w:L_x]=1$:
\[\ell_w=1+\sum_{v\leq w<x}\ell_v-\sum_{w<x}\ell_w=1+\sum_{v<w<x}\ell_v.\] 
Using induction and considering the Hasse quiver of the Bruhat order for the dihedral group
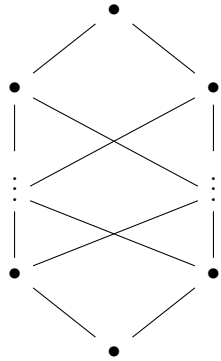
\[
\begin{tikzcd}
&\bullet\arrow[no head]{ld}\arrow[no head]{rd}\\
\bullet\arrow[no head]{d}\arrow[no head]{rrd}&&\bullet\arrow[no head]{lld}\arrow[no head]{d}\\
\vdots\arrow[no head]{d}\arrow[no head]{rrd}&&\vdots\arrow[no head]{lld}\arrow[no head]{d}\\
\bullet\arrow[no head]{rd}&&\bullet\arrow[no head]{ld}\\
&\bullet
\end{tikzcd}
\]
 we see that:
\begin{align*}
\ell_w&=1+\sum_{\substack{v<w<x\\v \text{ minimal}}} \ell_v + \sum_{\substack{v<w<x\\v \text{ not minimal}}} \ell_v=1+\sum_{\substack{v<w<x\\v \text{ minimal}}} 1 + \sum_{\substack{v<w<x\\v \text{ not minimal}}} 3^{\height(v)-1}\\
&=1+2(\height(x)-1)+\sum_{k=1}^{\height(x)-2}\sum_{\substack{v\\\height(v)=k}}\sum_{\substack{w\\v<w<x}}3^{k-1}\\
&=1+2(\height(x)-1)+\sum_{k=1}^{\height(x)-2}\sum_{\substack{v\\\height(v)=k}}2(\height(x)-k-1)3^{k-1}\\
&=1+2(\height(x)-1)+\sum_{k=1}^{\height(x)-2}4(\height(x)-k-1)3^{k-1}
\end{align*}
A straightforward calculation (which can also be performed using WolframAlpha) shows that the latter term equals the claimed term. 
\end{example}

It would be of interest to compute these multiplicities of projectives for other Coxeter groups. However, even given these multiplicities, there is no easy way to determine how the regular exact Borel subalgebra sits inside this Morita representative. The only known method so far to do this is to compute the $A_\infty$-structure on the $\Ext$-algebra of the standard modules. This seems to be a challenging task in light of the fact that even the computation of the dimension of extension groups between Verma modules is an active area of research, see e.g. \cite{KM22}. Note also that in small ranks, the quiver and relations (in the classical sense) for the blocks of category $\mathcal{O}$ have been determined in \cite{Str03}, for the case of dihedral groups, see \cite[Theorem 6.8]{Sau18}.




\bibliographystyle{emss}
\bibliography{publication}

\end{document}